\documentclass[a4paper,12pt]{article}
\usepackage[russian,english]{babel}
\usepackage[T1]{fontenc}
\usepackage{amsthm,enumerate,amsmath,amscd,amssymb}
\usepackage{graphicx}

   \usepackage{comment}

     \usepackage{blindtext} 
    \usepackage{xcolor}  
     \usepackage{hyperref}  

             \usepackage{amssymb}
\usepackage{amsthm}

\def\ove#1{\overline{#1}}    %

     \def\({\big(}                  
     \def\){\big)}      \def\e{\varepsilon}    
     \def\[{\left[}       
     \def\]{\right]}      \def\ffi{\varphi}

                                      \def\ot{\otimes}
     \def\<{\langle}                 \def\wh{\widehat}
     \def\>{\rangle}                 \def\wt{\widetilde}

\def\Id{\operatorname{Id}}

\newcounter{minutes}
\divide\time by 60
\newcounter{hours}
\multiply\time by 60
\addtocounter{minutes}{-\time}

\swapnumbers
\numberwithin{equation}{section}
\theoremstyle{plain}

\newtheorem{theorem}[equation]{Theorem}
\newtheorem{corollary}[equation]{Corollary}
\newtheorem{proposition}[equation]{Proposition}
\newtheorem{lemma}[equation]{Lemma}

\newtheorem{remark}[equation]{Remark}

\newtheorem{definition}[equation]{Definition}
\newtheorem{example}[equation]{Example}
\begin{document}



\begin{center}
{\bf \large Remarks on some results of G. Pisier  and P. Saab on convolutions\footnote{${ }$ The material of this paper has some intersections with the material of the paper \cite{OR21}.}
}
\end{center}

\begin{center}
{by \bf  Oleg Reinov
}
\end{center}

\begin{abstract}
{
A result of G. Pisier says that a convolution operator $\star f : M(G) \to C(G),$ where $G$ is a compact Abelian group, 
can be factored through a Hilbert space if and only if $f$ has the absolutely summable set of Fourier coefficients. P. Saab (2010) generalized this result
in some directions in the vector-valued cases. We give some further generalizations of the results of G. Pisier and P. Saab,
considering, in particular, the factorizations of the operators through the operators of Schatten classes in Hilbert spaces.
Also, some related theorem on the factorization of operators through the operators of the Lorentz-Schatten classes are obtained.
}
\end{abstract}









{\section{Introduction}}\label{0}

  In this note, we will be interested in problems related to the possibility of factorizations of various types
 of nuclear operators through operators in Hilbert spaces and their applications, in particular,
 in the theory of tensor products of Banach spaces.

Studying the behavior of the eigenvalues of nuclear operators, A. Grothen\-dieck  \cite{Gr55} based, in particular,
on the possibility of factorizations of  these operators through  the operators of Schatten classes $S_p,$
then applying the theory of operators in Hilbert spaces.
  Moreover, as is now clear, he obtained exact results in the scales of those spaces of operators that he studied.
  Let us give some examples of A. Grothendieck's ideas in applications of this technique (which, in particular,
  contributed to the appearance of this note).

 Recall that an operator $T: X\to Y$ in Banach spaces is called nuclear if it can be represented in the following form:
 \begin{equation}\label{T:N1}
Tx= \sum_{n=1}^\infty \mu_n \langle x'_n,x\rangle y_n, \ \text{ for }\ x\in X,
\end{equation}
  where $(x_n)\in X, (y_n)\in Y$ are bounded sequences,  $\mu:=(\mu_n)$ is an absolutely summable sequence of complex numbers
(it is clear that they can be considered real and non-negative).
 Every nuclear operator $T$ can be factorized in the following way:
  \begin{equation}\label{T:N1f}
 T=AB: X\overset{B}\to l^2\overset{A}\to Y,
\end{equation}
 where $A, B$ are bounded operators.
 Indeed, it is enough to put
 $Bx:= \{\sqrt \mu_n \langle x'_n,x\rangle\}\in l^2$ and $A(\alpha_n):= \sum \alpha_n \sqrt \mu_n y_n$ for $(\alpha_n)\in l^2.$
Let now $X=Y,$ consider, along with the operator  $AB,$ the operator $BA: l^2\to X\to l^2.$
If $e_k$ denotes $k$-th orth in $l^2$ then $Ae_k= \sqrt \mu_k y_k$ and
$$\langle BAe_k,e_m\rangle= \langle \sum_n\sqrt \mu_n \langle x'_n, \sqrt \mu_k y_k\rangle e_n, e_m\rangle = \sqrt \mu_m \langle x'_m, \sqrt \mu_k y_k\rangle ,$$
i. e. $\langle BAe_k,e_m\rangle = \sqrt \mu_m \langle x'_m, y_k\rangle   \sqrt \mu_k$ and thus $\sum_k ||BAe_k||^2<\infty.$
Therefore, $BA$ is the Hilbert-Schmidt operator.
Further, since the (complete) sequence of eigenvalues of the operator $BA$ (counted with their algebraic multiplicities)
lies in $l^2,$ the complete sequence of all eigenvalues of the nuclear operator $ T = AB $ also belongs to $l^2.$

As A. Grothendieck \cite{Gr55} noted, this result is sharp, and   the continuous on the unit circle 
Carleman function whose Fourier coefficients lie  in $l^2$  and do not lie in any space $l_p$ for 
$p<2$ gives  an example of this. 
 Indeed, the convolution operator
 with this function considered in the space of all continuous functions on the circle $ \mathbb T $
 is nuclear, and the complete set of its eigenvalues coincides with the sequence of Fourier coefficients
 of the Carleman function. At the same time, this example shows that for an arbitrary nuclear operator
 in Banach spaces, the factorization through a continuous operator in a Hilbert space (through an operator
 of the class $ S_\infty; $ here, through the identity operator) is the best possible among the factorizations
 of operators though operators of the Schatten classes $S_p.$
 Indeed, otherwise (e. g., in the case of a possible factorization through the $S_p$-operator with $ p\in (0,\infty)$),
 the above reasoning with the Hilbert-Schmidt operator would allow us to conclude that
 the eigenvalues of the nuclear operators lie in some $ l^q $ for $ q<2 $
 (namely, for $ q $ from the relation $ 1/2 + 1/p = 1/q; $ we use the inclusion $ S_p \circ S_2 \subset S_q).$
 Note, however, that each nuclear operator can be factorized through a compact operator
 in Hilbert space. To see this, it is enough to split off from the non-negative sequence $ \mu $
 a piece (sequence) $ \nu: = (\nu_n) $ so as to obtain the relations
 $ \nu_n \to \infty $ and $ \sum \mu_n \nu_n <\infty. $
 Correcting the definitions of the operators $ A $ and $ B $ in (\ref {T:N1f}) respectively,
 we obtain the factorization  of $ T $ through a diagonal operator in $ l^2 $ with a diagonal tending
 to zero (for example, through $\Delta:= (\nu_n).$)

The situation is somewhat more complicated with the so-called $ p $-nuclear operators for $ 0 <p <1 $
(considered for the first time by A. Grothendieck, but under the name "Operateurs de puissance p.eme sommable").
 The definition of a $ p$-nuclear operator is similar to the above definition of a nuclear operator,
but in the relation (\ref{T:N1}) we consider the sequence $ (\mu_n) $ from the space $ l^p. $
Each $p$-nuclear operator $ T: X \to Y $ admits a factorization of the form (\ref {T:N1f}).
Moreover, one can factorize $ T $ not only through a compact operator, but also through an operator
from the class $ S_q $ for the exponent $ q $ with $ 1/q = 1/p-1 $ (see~\cite{Gr55}, Chap. 2, p . 11).   
As above, this implies the corresponding result on the distribution of the eigenvalues of $p$-nuclear operators:
the eigenvalues of $ p$-nuclear operators lie in $ l_q, $ where $ 1/q = 1/p-1/2. $
It should be noted, however, that this result is not final. In the scale of spaces $ l_p $, it is exact
(see, for example, \cite{K86}). But if we go to the scale of Lorentz spaces $ l_{p, r}, $ then
the exact result looks like this (\cite{K86}): for any $ p\in (0,1) $ the sequence of eigenvalues of any
$p$-nuclear operator lies in $ l_{q, p}, $ where $ 1/q = 1/p-1/2. $

In the first case, A. Grothendieck's remark on the sharpness of his statement on the eigenvalues
of nuclear operators (using the example of Carleman) is unimprovable in the scale of Lorentz spaces.
Perhaps the use of the Carleman function confirms this fact, but I do not know this, and I did not
check if this function lies in any Lorentz space $ l_{2, s} $ for $ s<2. $
Therefore, let us use, e. g., the Kahan-Katznelson-de Leeuw theorem \cite{KKL}:
for any sequence $ a\in l_2(\mathbb Z) $ there exists a function $ F_a \in C(\mathbb T)$
such that for all $ j\in \mathbb Z, \, |\hat F_a(j)| \ge |a_j|. $
To obtain the exactness of Grothendieck's theorem (in the scale of Lorentz sequence spaces), it is sufficient to
take any sequence $ a\in l_2 (\mathbb Z) \setminus \cup_{s<2} l_{2, s} $ and the corresponding function $ F_a $
and consider, following Grothendieck, the convolution operator with this function in the space $C(\mathbb T).$
Note that a simpler example could be used, namely the example from the book
\cite{K86} (Example 2 .b.14, p. 107), in which the sharpness
(in the scale of Lorentz sequence spaces)
of the Grothendieck's theorem is explicitly obtained:
for any sequence $ \sigma \in l^2$ there exists a nuclear operator $ T $ such that the sequence
$ \sigma \in l^2$ is a subsequence of the sequence
the eigenvalues of the operator $T. $ Finally, we can use the stronger (final) result of R. Kaiser and J. Rutherford \cite{KR}:
any sequence $ a\in l^2 $ is exactly a sequence of eigenvalues of some nuclear operator
(zero terms of sequences are not considered).

All  that was said above shows that for a nuclear operator to be factorized through a compact operator
$(S_\infty$-operator)
in a Hilbert space is the best that one can get answering a question "What are the  exponents
$p,q\in (0,\infty]$ for which  every nuclear operator can be factored through an $S_{p,q}$-operator?"
(Lorentz--Schatten class; see below).

In the second case, where we consider the $p$-nuclear (or more generally, $(p,q)$-nuclear) operators,
we can use the same factorization ideas as above for the case $p=1.$
It can be seen that every $p$-nuclear operator $T$ can be factored through an operator from $S_{q}(H),$
where $1/q=1/p-1$ (so, $q=\infty$ if $p=1).$ It follows from such a factorization that the eigenvalue
sequence of $T$ lies in $l_r$ with $1/r=1/p-1/2.$ In the scale $S_p$ of Schatten classes the last result is best possible.
The same can be said about the factorizations of $p$-nuclear operators through an $S_q$-operators.
However. if we consider the scale s of Lorentz sequence spaces $l_{r,s}$ and $S_{r,s}$ of operators of Lorentz-Schatten classes,
then the questions are not so clear. Recall that an operator $T\in L(X,Y)$ is said to be $(r,s)$-nuclear,
where $0<r,s\le1,$ if it admits a representation
(\ref{T:N1} with ${(\mu_n)}\in l_{r,s}.$

 On the one hand,  H. K\"onig 
 showed in \cite{K80},
 that the eigenvalues of the $p$-nuclear operators $(0<p<1)$
 lie in the Lorentz space $l_{r,p},$ where $1/r=1/p-1/2$ and this result is sharp (see~\cite{K86}, p. 126).


On the other hand, the eigenvalues result do not give a possibility to find out whether
the above result on the factorization of $T$ through an $S_q$-operator is the best one in the scale
$S_{r,s}.$ We need to proceed in another way.

A result of G. Pisier \cite{Pis85} gives us a possibility to get one more negative answer to the question considered in \cite{Rei17}
on the product of two nuclear operators (see below Corollary \ref{NN}).
G. Pisier has shown that if a convolution operator
$$\star f : \, M(G)\to C(G),$$
where $G$ is a compact Abelian group, $M(G)=C(G)^*$ and $f\in C(G),$
can be factored through a Hilbert  space, then $f$ has
the absolutely summable set of Fourier coefficients.
It is clear that the condition
 "the convolution operator  ... can  be factorized through a Hilbert space"
 $$
\star f: M(\mathbb T) \to H \to C(\mathbb T)
$$
is  the same as the condition
"the operator $\star f$ can be factored through a bounded operator $U$ in a Hilbert space":
$$
\star f: M(\mathbb T) \to H \overset U\to H \to C(\mathbb T).
$$
We are going to generalize this result, so let us give some notes about it.

 Let $S(H)$ be an ideal in the algebra $L(H)$ of all bounded operators in $H$ (e.g., the
ideal of compact operators). What is the condition on the set $\{\hat f(n)\}$ that gives
a possibility to factorize the operator $\star f$ through an operator from $S(H)?$
$$
\star f: M(\mathbb T) \to H \overset {U\in S}\longrightarrow H \to C(\mathbb T).
$$
We present some generalizations of the result of G. Pisier, giving answers to the question for the ideals
$S_{p}(H)$ of operators from the Schatten classes
(operators, whose singular numbers are in the sequence space $l_{p})$ and for general compact Abelian groups $G:$
$$
\star f: M(G) \to H \underset{S_p(H)}\longrightarrow H \to C(G),\, \ f\in C(G).
$$
Moreover, we will consider even convolution operators in vector-valued function spaces, generalizing
the result of G. Pisier and a result of P. Saab \cite{Sa8}, Theorem 4.2, where
it was shown that the Pisier's techniques in the scalar case can be
extended to the vector-valued case (factorizations of a vector valued convolutions through Hilbert spaces).
We will get also two theorems which are very close to
to some generalizations of
main theorem from \cite{Sa8}.
generalizing also main results of P. Saab from \cite{Sa8}.

{\section{Preliminaries}}\label{1}
All the spaces  $X,Y,Z,W, \dots$
are Banach, $x, x_n, y, y_k, \dots$ are elements of spaces $X, Y,\dots$ respectively.
All linear mappings (operators) are continuous; as usual, $X^*, X^{**}, \dots$
are Banach duals (to $X$), and $x', x'', \dots$ (or $y', \dots)$ are the functionals on $X, X^*,\dots$
(or on $Y,\dots).$ By $\pi_Y$ we denote the natural isometric injection of $Y$ into its second dual.
If $x\in X, x'\in X^*$ then $\langle x,x'\rangle=\langle x',x\rangle=x'(x).$
$L(X,Y)$ stands for the Banach space of all linear bounded operators from $X$ to $Y.$
We always consider the space $X$  as the  subspace
$\pi_X(X)$ of its second dual $X^{**}$ (denoting, if needed, by $\pi_X$ the canonical injection).

\subsection{Analysis on Groups}

We refer to \cite{Rud} on general topics of this subsection and to \cite{DiVM} for the information on
vector-valued function spaces and vector measures.

Let $G$ be a compact Abelian group,
$m$ be a Haar measure on $G$ 
(i.e. the unique translation invariant normalized regular Borel measure, or,
what is the same, Radon probability),
$\Gamma$ be the dual group of $G,$ i.e., the group of all characters on $G$
(multiplicative continuous complex functions $\gamma$ 
so that $|\gamma(t)|=1$ for all $t\in G).$   
Note that $\Gamma$ is discrete.
$C(G)$ is the Banach space of all continuous (complex-valued) functions on $G$
with the natural uniform norm: if $\ffi\in C(G,)$ then
$$
||\ffi||_\infty := \sup_{t\in G} |\ffi(t)|.
$$
$L_p(G),$ $1\le p<\infty,$ --- the Banach space of all $(m$-equivalent classes of)
absolutely $p$-summable Borel functions on $G,$  
$$
||\ffi||_p := (\int_G |\ffi|^p\, dm)^{1/p}<\infty \, \text{ for }\,  \ffi\in L_p(G).
$$
$M(G)$ is the Banach space of all (complex-valued) finite regular Borel measures on $G$
with  the variation norm $|\mu|(G)$ (or, what is the same, with the norm induced from $C^*(G)$ 
by the Riesz Representation Theorem).

 If $f\in L_p(G),$ $1\le p<\infty,$ and $\mu\in M(G),$ then
 $$
 f\star \mu(g):= \int_G f(g-h)\, d\mu(h)\, \text{ for }\, g\in G,
 $$
 $$
 || f\star \mu||_p\le ||f||_p\, ||\mu||.
 $$
 If $f\in C(G),$ then
  $$
 f\star \mu(g)\in C(G)\, \text{ and } \,
 || f\star \mu||_\infty\le ||f||_\infty\, ||\mu||.
 $$
If $f\in L_1(G),$   
and $\mu\in M(G),$ then the Fourier transform of $f$ and $\mu$ are defined by
 $$
 \hat f(\gamma):= \int_G \overline{\gamma(h)} f(h)\, dm(h)\, \text{ for }\, \gamma\in \Gamma;\
 $$
 (maps $L_1(G)\to C_0(\Gamma), \,  || \hat f||_\infty\le ||f||_1)$ and
   $$
 \hat \mu(\gamma):= \int_G \overline{\gamma(h)}\, d\mu(h)\, \text{ for }\, \gamma\in \Gamma.
 $$
 Here $C_0(\Gamma)$ is the subspace of $C(\Gamma),$ consisting of all functions
which vanish at infinity.


\subsection{Tensor products and  summing operators}

We refer to \cite{DiVM, Gr55, Gr56, SchatTensors} on tensor products of Banach spaces and to \cite{PiOP, DiASO} for the information on
$p$-absolutely summing operators.

\subsubsection{Tensor product and integral operators}
For Banach spaces $X,Y,$ denote by $F(X,Y)$ the linear subspace of the space $L(X,Y,)$ consisting
of all finite rank operators. Algebraic tensor product $X^*\otimes Y$ will be identify with
the linear space $F(X,Y):$ every tensor element $z:=\sum_{n=1}^N x'_n\otimes y_n$ can be
considered as an operator $\widetilde z(\cdot):=\sum_{n=1}^N \langle x'_n,\cdot\rangle  y_n.$
Also, $X\otimes Y$ can be considered as a subspace of the vector space $F(X^*, Y)$
(namely, as vector space of all linear weak${{}^*}$-to-weak continuous
finite rank operators). We can identify also the tensor product (in a natural way)
with a corresponding subspace of $F(Y^*,X).$ If $X=W^*,$ then
$W^*\otimes Y^{**}$ is identified with $F(X,Y^{**})$ (or with $F(Y^*,X^*).$

The projective norm of an  element $z\in X\otimes Y$ is defined as
$$
||z||_{\land}:=\inf \{ \sum_{n=1}^N ||x_n||\,||y_n||:\ z=\sum_{n=1}^N x_n\otimes y_n,\ (x_n)\subset X, (y_n)\subset Y\}.
$$
The completion of the normed space $(X\otimes Y, ||\cdot||_\land)$ is called the projective tensor
product of Banach spaces $X$ and $Y$ and denoted by $X\widehat\otimes Y.$ Every element
can be written in the form
\begin{equation} \label{tp1}
z=\sum_{n=1}^\infty x_n\otimes y_n\ \text{ with } \sum_{n=1}^\infty ||x_n||\,||y_n||<\infty.
\end{equation}
Note, that $X\widehat\otimes Y= Y\widehat\otimes X$ under natural identification.
Every element $z$ of $X\widehat\otimes Y$ generates an operator $\widetilde z: X^*\to Y:$
If $z$ has a representation \ref{tp1}, then $\widetilde z(x'):= \sum_{n=1}^\infty \langle x_n,x\rangle  y_n.$

A linear functional "trace" is defined on each tensor product $z\in X^*\otimes X:$
If $z=\sum_{n=1}^N x'_n\otimes x_n,$ then $\operatorname{trace}\ z=\sum_{n=1}^N \langle x'_n, x_n\rangle $ and
the last sum does not depend on a representation of $z.$
This functional has a unique extension to the completion $X^*\widehat\otimes X$ and its value at an element
$z\in X^*\widehat\otimes X$ is denoted again by $\operatorname{trace}\ z.$ If $z=\sum_{n=1}^\infty x_n\otimes x_n,$ then
$\operatorname{trace}\ z= z=\sum_{n=1}^\infty \langle x'_n, x_n\rangle .$

A dual space of the tensor product $X\widehat\otimes Y$ is $L(Y,X^{*})$ with duality
defined by
$$
\langle T, z\rangle := \operatorname{trace}\ T\circ z = \sum_{n=1}^\infty \langle x_n, Ty_n\rangle ,\ z\in X\widehat\otimes Y, T\in L(Y,X^*).
$$
Here, $T\circ z$ is an element $\sum_{n=1}^\infty x_n\otimes Ty_n\in X\widehat\otimes X^*.$
In particular, $(X^*\widehat\otimes Y)^*=L(Y, X^{**})=L(X^*,Y^*).$

If $A\in L(X,W),$ $B\in L(Y,G)$ and $x\otimes y\in X\otimes Y,$ then  a linear map
  $A\otimes B: X\otimes Y\to W\otimes G$ is defined by
  $A\otimes B((x\otimes y):= Ax\otimes By$ (and then extended by linearity). Since
 $\widetilde{A\otimes B(z)}= B\widetilde z A^*$ for $z\in X\otimes Y,$ we can use notation
$B\circ z \circ A^*\in W\otimes G$ for $A\otimes B(z).$

There is another natural norm on the tensor product $X\otimes Y,$ namely, the norm
induced from $L(X^*,Y),$ that is the uniform norm. The completion of $X\otimes Y$
with respect to this norm coincides with the closure of $X\otimes Y$ in $L(X^*,Y),$
is denoted by $X\widetilde\otimes Y$ and is called the injective tensor product of $X$ and $Y.$
 In particular, the injective tensor product $X^*\widetilde\otimes Y$ is exactly the closure
 of all finite rank operators in $L(X,Y)$ and contained in the Banach space $K(X,Y)$
 of all compact operators from $X$ to $Y.$

 The dual space to $X\widetilde\otimes Y$ can be identify with so-called integral operators
 from $Y$ to $X^*.$ We will use the following definition of an integral operator in
 Banach spaces: An operator $T:Z\to W$ is said to be integral
 (they say also "integral in the sense of Pietsch") if there  exist a compact space $K,$ a probability
 measure $\mu\in C^*(K)$ and two bounded operators $A:Z\to C(K)$ and $B: L_1(K,\mu)\to W$
 so that $T$ admits the following factorization:
 $$
 T=BjA: Z\overset A\to C(K)\overset j\hookrightarrow L_1(K,\mu)\overset B\to W
 $$
 where $j$ is a natural inclusion.
 With the norm $i(T):=\inf ||A||\,||B||$ the space $I(Z,W)$ of all integral operators
 is Banach. For any $z=\sum_{n=1}^N x_n\otimes y_n\in X\otimes Y$ and $V\in I(Y,X^*)$ the composition
 $V\circ z$ lies in $X\otimes X^*$ and $||V\circ z||_\land\le ||\widetilde z||i(V).$ Thus $V$
 generates a linear continuous map from $X\widetilde\otimes Y$ into $X\widehat\otimes X^*,$ the trace
 of $V\circ A$ is well defined for every $A\in X\widetilde\otimes Y$ and $|\operatorname{trace}\ V\circ A|\le ||A||i(V).$
 The linear continuous functional $\operatorname{trace}\ V\circ\cdot$ defines a duality between the spaces
 $X\widetilde\otimes Y$ and $I(Y,X^*)$ and the last space is the dual to the injective tensor product $X\widetilde\otimes Y$
 with respect to  this duality.

 Let us mention that the above norms in $\widehat\otimes$ and in $\widetilde\otimes$ are the greatest and least crossnorms
 respectively (see, e. g., \cite{DiVM}, p. 221).
 Projective and injective tensor products of several Banach spaces can be defined by induction.

 Two  important notion in connection with the just introduced notions:
 They say that a Banach space $X$ has the approximation property if for every Banach space $Y$
 the natural mapping $Y^*\widehat\otimes X\to L(Y,X)$ is injective; $X$ has  the metric approximation
 property if for every Banach space $Y$
 the natural mapping $Y^*\widehat\otimes X\to I(Y,X^{**})$ is an isometric embedding.
 Such spaces as $L_p(\mu), C(K),$ $M(G)=C^*(G)$ and all their duals have the metric approximation
 property \cite{Gr55}.

\subsubsection{Absolutely summing operators}
A series $\sum_{n=1}^\infty x_n$ in a Banach space $X$ is unconditionally convergent if
for every permutation $\pi: \mathbb N\to \mathbb N$ of the natural numbers
the series $\sum_{n=1}^\infty x_{\pi(n)}$ is convergent too.
It is the same as the convergence of the series $\sum_{n=1}^\infty b_n x_n$ for
every bonded sequence $(b_n)$ (see \cite{DiASO}, 1.9). An operator $T: X\to Y$ is said
to be absolutely summing if it takes any unconditionally convergent series in $X$ to a absolutely
convergent series in $Y.$ A famous Dvoretzky-Rogers theorem says that the identity map in $X$
is absolutely summing iff the space $X$ is finite dimensional (see, e. g.,\cite{DiASO}, 1.2).

Example: An inclusion $j: C(K)\to L_1(K,\mu),$ where $\mu$ is a probability measure on a compact
set $K.$

An operator $T:X\to Y$ is said to be $2$-absolutely summing if there is a constant $C>0$ such that for every
finite sequence $(x_n)_1^N\subset X$ one has
$$
\sum_{n=1}^N ||Tx_n||^2\le C^2\sup_{||x'||\le1}\left|\sum_{n=1}^N|\<x_n, x'\>|^2\right|.
$$
The set $\Pi_2(X,Y)$ of all such operators is Banach with a norm $\pi_2(T)=\inf C.$

Examples: 1)\, An inclusion $j: C(K)\to L_2(K,\mu),$ where $\mu$ is a probability measure on a compact
set $K.$ 2)\, $\Pi_2(H,H)=S_2(H,H).$ 3)\, Any operator from $C(K)$ to $M(K)=C^*(K)$ is  $2$-absolutely summing.

Generally, let $0<r<\infty.$
An operator $T:X\to Y$ is said to be $r$-absolutely summing if there is a constant $C>0$ such that for every
finite sequence $(x_n)_1^N\subset X$ one has
$$
\sum_{n=1}^N ||Tx_n||^r\le C^r\sup_{||x'||\le1}\left|\sum_{n=1}^N|\<x_n, x'\>|^r\right|.
$$
The set $\Pi_r(X,Y)$ of all such operators is (quasi)Banach with a (quasi)norm $\pi_r(T)=\inf C.$



\subsection{Lorentz-Schatten classes of operators in Hilbert spaces}

The Lorentz-Schatten class $S_{p,q},$ $0<p,q<\infty,$ considered for the first time by H. Triebel in \cite{Tri67},   
can be defined in the following way.
Let $U$ be a compact operator in a Hilbert space $H$ and $(s_n)$  is the sequence of its
singular numbers  (see, e. g., \cite{EigPie}, 2.1.13).
An operator  $U$ belongs to the space $S_{p,q}(H),$
if $(s_n)\in l_{p,q}.$
(see, e. g., \cite{EigPie}, 2.11.15).
The space
$S_{p,q}(H)$ has a natural quasi-norm
$$\sigma_{p,q}(U)=||(s_n)||_{p,q}=\left(\sum_{n=1}^\infty n^{(q/p)-1} s_n^q\right)^{1/q}.$$
If $p=q,$ then $S_{p,p}$ coincides with the class $S_p$ (with a quasi-norm $\sigma_p).$
 Let us mention that, for $p,q\in (0,1],$ we have the equality
$N_{p,q}(H)=S_{p,q}(H)$ (see, e. g., \cite{HiPi}) and
the inclusions
$S_{p,q}\subset S_{p,q'},$ if $0<p<\infty$ and $0<q\le q'<\infty$ or $S_{p,q}\subset S_{p',q'}$ if
$ 0<p<p'<\infty, 0<q,q'<\infty$
(see~\cite{Tri67}, Lemma 2) and
$$
S_{p,q}\circ S_{p',q'}\subset S_{s,r},\ 1/p+ 1/p'=1/s,\, 1/q+1/q'=1/r.
$$
Moreover, if $V\in S_{p,q}$ and $U\in S_{p',q'},$ then
$\sigma_{s,r}(UV)\le 2^{1/s}\sigma_{p',q'}(U)\, \sigma_{p,q}(V)$
(see~\cite{Pie80}, p. 155).
In the case where $p=q, p'=q',$ one has the constant $1$ instead of  $2^{1/s}$ in the last inequality
\cite{Horn}, \cite{EigPie}, p. 128, \cite{BirSol}, p.262.   

Examples of the $S_{p,q}$-operators are the diagonal operators $D$
in $l_2$ with  the diagonals $(d_n)$ from $l_{p,q};$ in such cases we write $D=(d_n).$

Given two complex Hilbert spaces $H_1$ and $H_2$, we denote by $H_1\otimes_2 H_2$
the completion of the tensor product $H_1\otimes H_2$ with respect to the
natural scalar product.


\section{On a Pisier's result}\label{2}

In this section we  are going to prove some generalizations of the Pisier's theorem mentioned
in Introduction to the cases of $S_{p}$-factorizations of operators for  scalar cases.
Some applications are given

\begin{definition}
An operator $T\in L(X,Y)$ is said to be $r$-nuclear, where $0<r\le1,$ if it admits a representation
\begin{equation} \label{Nsr}
Tx= \sum_{n=1}^\infty \mu_n \langle x'_n,x\rangle  y_n, \ \text{ for }\ x\in X,
\end{equation}
where $(x_n)\in X, (y_n)\in Y$ $||x_n||\le1, ||y_n||\le1$ and $(\mu_n)\in l_{r}.$
We put $\nu_{r}(T):=\inf ||(\mu_n)||_{l_{r}},$ where the infimum is taken over all possible
factorizations of $T$ in the form (\ref{Nsr}).
\end{definition}
It is clear that we can assume that ${\mu_n}$  are real and non-negative.
With the quasi-norm $\nu_{r},$ the space $N_{r}(X,Y)$ of all   $r$-nuclear operator
from $X$ to $Y$ is a complete quasi-normed space. We need the following well known fact.
The proof is given for completeness.

\begin{proposition}\label{Pr1}
If $T\in N_{r}(X,Y)$ $(0<r\le1),$
then $T$
can be factored through an operator from $S_{v}(H),$
where $1/v=1/r-1.$
Moreover, $\gamma_{S_{v}}(T)\le \nu_{r}(T)$
\end{proposition}

\begin{proof}
$ T: X \to Y $ admits the following factorization:
\begin{equation}\label{Eq1}
  T:\, X\overset{W}\to l_\infty\overset{\Delta}\to l_1\overset{V}\to Y,
   \end{equation}
  where $||V||=||W||=1$ and $\Delta$ is a diagonal operator with a diagonal $(d_n)\in l_{r}.$
  Indeed, it is enough to put
  $Wx:= (\langle x'_k,x\rangle),$ $V(\alpha_n):= \sum \alpha_n y_n$ and $\Delta (\beta_n):= (d_n\beta_n)$
  (where $d_n:=\mu_n).$
rewrite the  factorization (\ref{Eq1}) as follows:
  \begin{equation}\label{Eq2}
  T:\, X\overset{W}\to l_\infty\overset{\Delta_1}\to l_2
     \overset{\Delta_0}\to l_2
\overset{\Delta_2}\to  l_1\overset{V}\to Y,
  \end{equation}
where $\Delta_1:=(\sqrt{d_n^r}),$   $\Delta_2:=(\sqrt{d_n^r})$ and
$\Delta_0:=(d_n^{1-r}).$

Suppose that $\varepsilon>0$ and in the factorization (\ref{Eq1}) $||V||=||W||=1$ and $||(d_n)||_{l_{r}}\le (1+\varepsilon) \nu_{r}(T).$
Then
\begin{eqnarray}
||\Delta_2|| & = & ||\Delta_1||\le \pi_2(\Delta_1)\le ||(\sqrt{d_n^r})||_{l_2}  \label{Eq3} \\
{} & = & ||(d_n^r)||^{1/2}_{l_1}\le [(1+\varepsilon) \nu_{r}(T)]^{r/2}. \nonumber
  \end{eqnarray}
Also
$\Delta_0\in S_{v}(l_2),$ where  $1/v=1/r-1.$
Moreover, since  $1-r=r/v,$ 
we have
    \begin{eqnarray}
 \sigma_{v}(\Delta_0) & = & \left(\sum d_n^{1-r}]^v\right)^{1/v}  \label{Eq4} \\ 
   {} & = &  \left(\sum [d_n]^{r}\right)^{1/v}\le [(1+\varepsilon) \nu_{r}(T)]^{1-r}.  \nonumber
   \end{eqnarray}
\end{proof}


\begin{remark}\label{Rem1}
As a matter of fact we see that $T=A\Delta_0B,$ where $A\in \Pi_2^{dual}(l_2,Y),$
$B\in \Pi_2(X,l_2)$ and $\Delta_0\in S_v(l_2, l_2).$
\end{remark}

\begin{example}\label{Ex1}
Let $T\in N_r(X,Y)$ and $U\in N_1(Y,X),$ where $0<r\le1.$
We have the diagram
$$
UT:\, X\overset{B_2}\to l_2\overset {\Delta_0}\to l_2\overset{A_2}\to Y
 \overset{B_1}\to l_2\overset{A_1}\to X,
$$
where $B_2, B_1\in\Pi_2$ and $\Delta_0\in S_v(l_2),$ $1/v=1/r-1.$
Eigenvalues of $UT$ are the same as ones of the operator $V:= B_2A_1B_1A_2\Delta_0:$
$$
l_2\overset {\Delta_0}\to l_2\overset{A_2}\to Y
 \overset{B_1}\to l_2\overset{A_1}\to X \overset{B_2}\to l_2.
$$
Since $B_2, B_1\in\Pi_2,$ we have $ B_2A_1B_1A_2\in S_1(l_2).$
Therefore, $V\in S_r$ $(1+1/v=1/r).$
Thus, the sequence of all eigenvalues of $UT$ lies in $l_r.$
\end{example}

\begin{remark}
It can be shown that if $U\in N_p(Y,X)$ and $0<r,p<1$ in Example \ref{Ex1}, then eigenvalues of $UT$
belong to the Lorentz space $l_{s,q}$ with $1/s=1/r+1/p-1$ and $1/q=1/r+1/p.$  Example \ref{Ex1}
shows the specificity of the particular case $p=1.$
\end{remark}
We are going to show that the results from Proposition \ref{Pr1} and Example \ref{Ex1} are sharp.
For this we need the following first generalization of the Pisier result, mentioned
in Introduction.


\begin{theorem}\label{ThP1}
Let $f\in C(G),$  $0<s\le1$ and  $1/r=1/s-1.$ Consider a convolution operator  
$\star f: M(G)\to C(G).$
The set  of  Fourier coefficients $\hat f$ belongs to $l_{s}$
if and only if the operator $\star f$ can be factored through a
Schatten $S_{r}$-operator in a Hilbert space.
\end{theorem}

\begin{proof}
1)\,
Let there exists $U\in S_{r}(H)$ such that 
$$
\star f=AUB:\, M(G)\overset B\to H \overset U\to H \overset A\to C(G).
$$
If $j: C(G)\hookrightarrow M(G)$ is a natural injection, then the Fourier coefficients of  $f$ are the eigenvalues of
the operator $AUBj: C(G)\to M(G)\to C(G).$
Consider a diagram
$$
C(G)\overset j\hookrightarrow M(G) \overset B\to H \overset U\to H \overset A\to C(G)\overset j\hookrightarrow M(G)\overset B \to H.
$$
The operators $AUBj$ and $BjAU$ have the same sequences of eigenvalues.
Since $B\in \Pi_2(M(G),H)$ $j: C(G)\hookrightarrow L_2(G)\hookrightarrow M(G)\in \Pi_2(C(G), M(G))$ and $U\in S_{r},$
we get that
$$
({*})\ \ BjAU\in S_{r}\circ S_1\subset S_{s},
$$
where $1/s=1+1/r.$ Therefore, the eigenvalues of $AUBj$ lie in $l_s.$ So
$\{\hat f(\gamma)\}\in l_{s}.$

2)\,
Suppose that
$\{\hat f(\gamma)\}\in l_{s},$ where $1/s=1+1/r.$ Let $\{c_{\gamma_n}=|\hat f(\gamma_n)|\}$
(non-zero part of $\hat f).$
Consider the operators
$B: M(G)\to L_2(G),$ $U: L_2(G)\to L_2(G)$ and $A: L_2(G)\to C(G),$
defined by
$$
B\mu:= \sum_n\hat\mu(\gamma_n) c_{\gamma_n}^{s/2} \gamma_n,\
 \, U\varphi:= \sum_n \hat{\varphi}(\gamma_n) c_{\gamma_n}^{1-s} \gamma_n
$$
and
$$
A\psi:= \sum_n \hat{\psi}(\gamma_n) \text{  sign} {\hat f(\gamma_n)} c_{\gamma_n}^{s/2} \gamma_n
$$
The operators are well defined since
the series
$$\sum \hat{\psi}(\gamma_n) \text{  sign} {\hat f(\gamma_n)} c_{\gamma_n}^{s/2} \gamma_n$$
is absolutely convergent. 

We have
$$
\star f=AUB:\, M(G)\overset B\to l_2 \overset U\to l_2 \overset A\to C(G),
$$
where $A, B$ are bounded and $U$ is from $S_{r}(l_2).$
\end{proof}

\begin{corollary}\label{CP1}
Let $f\in C(G),$  $0< s\le1$ and  $1/r=1/s-1.$
 For the convolution operator
$\star f: M(G)\to C(G),$
The following are equivalent:

1.\,
$\star f$ can be factored through a
Schatten $S_{r}$-operator;

2.\,
$\hat f\in l_{s};$

3.\,
$\star f\in N_{s} (M(G), C(G)).$
\end{corollary}

\begin{proof}
3)$\implies$1) is valid for any $s$-nuclear operator (Proposition \ref{Pr1}).

2)$\implies$3).
It is enough to consider the diagram
$$\star f=AUB:\, M(G)\overset B\to l_\infty \overset \Delta\to l_1 \overset A\to C(G),$$
where
$B\mu:= \{\hat\mu(\gamma)\},$
$\Delta \{a_\gamma\}:= \{|\hat f(\gamma)| a_\gamma\},$
$A \{b_\gamma\}:= \sum_\gamma \text{ sign} \hat f(\gamma) b_\gamma \gamma.$
\end{proof}


It follows from the above corollary  that the result of Proposition \ref{Pr1} is sharp.

We now give an application to the products of two nuclear operators .
A. Grothendieck \cite{Gr55} proved that the eigenvalue sequence of a product of
two nuclear operators is absolutely summable. The following corollary shows
that the result is sharp in the scale $l_{r,s}.$

\begin{corollary}\label{NN}
There exist two nuclear operators $t_1$ and $t_2$ whose product has the eigenvalues in
$l_1\setminus \cup_{s<1} l_{1,s}.$ 
\end{corollary}
\begin{proof}
Let $f\in C(G)$ with $\widehat f\in l_1\setminus \cup_{s<1} l_{1,s},$
$T_1=\star f: M(G)\to C(G)$ and $T_2:C(G)\to M(G)$ be a natural inclusion. 
By Corollary \ref{CP1}, the operator $T_1$ is nuclear. By using a natural factorization
of $T_1$ through a diagonal operator from $l_\infty$ to $l_1,$ we can represent $T_1$
as a product $kt_1$ of a nuclear operator $t_1: M(G)\to l_1$ and a compact operator $k: l_1\to C(G).$
Since $T_2$ is an  integral operator, the product operator $T_2k$ is nuclear. Put $t_2:=T_2k.$
Then, $t_1t_1=T_2T_1.$ The  sequence of the eigenvalues of $t_2t_1$ coincides
with $\widehat f.$
\end{proof}

\begin{remark}
Of course, the assertion of the corollary follows implicitly from the Pisier's theorem.
Also, note that Corollary \ref{CP1} gives us one more proof of the sharpness of the Grothendiek's theorem
about 
his factorization result for nuclear operators
(take $s=1$ and $f$ with $\widehat f\in l_1\setminus \cup_{s<1} l_{1,s}$).
\end{remark}

A slightly more general corollary (compare Example \ref{Ex1}):

\begin{corollary}
For every $r\in (0,1)$ there exist two nuclear operators $t_1\in N_r$ and $t_2\in N_1$ whose product has the eigenvalues in
$l_r\setminus \cup_{s<r} l_{r,s}.$ 
\end{corollary}
\begin{proof}
Take $f\in C(G)$ with $\widehat f\in l_r\setminus \cup_{s<r} l_{r,s}$ and proceed as in the previous proof.
\end{proof}

\begin{remark}
The same can be done for the products of type $$N_1N_pN_1N_q\cdots N_1N_r.$$
\end{remark}


\section{Around Saab's theorem}\label{4}

We will get here two theorems which are in a sense 
 generalize the
main Saab's theorem from \cite{Sa8}.


Following \cite{Sa8},
for $\ove f\in C(G,X),$ define a convolution operator $\star \ove f:  M(G)\to C(G,X)$ as
$$
\star f(\mu)=\ove f*\mu(s):=\int_G \ove f(s-t)\, d\mu(t)\in X.
$$



\begin{definition}\label{D1}
$T\in L(X,Y)$ possesses the property $(\mathfrak l)$ if
for every $\ove f\in C(G,X)$ such that $\star \ove f$ can be factored through a Hilbert space 
$$
\star \ove f: M(G)\overset A\to H\overset B\to C(G,X)
$$
the family $T\widehat{\ove f}$ is absolutely summable.
\end{definition}


\begin{theorem}\label{ThS2}
For functions $\ove f\in C(G,X),$
consider the convolution operators
$\star \ove f: M(G)\to C(G,X)$ and let 
$T\in L(X,Y).$ Consider the following assertions:

1)\,
$T\in L(X,Y)$ possesses the property $(\mathfrak l).$

2)\,
$T\in \Pi_2(X,Y).$

3)\, If $d\in l_1^{weak}(X)\cap l_2(X),$ then $Td\in l_1(Y).$ 

We have:
$1) \implies 2)$  and $3) \implies 1).$
\end{theorem}
\begin{proof}
$1) \implies 2).$\,
Fix any infinite sequence $(\gamma_n)$ of
distinct elements of $\Gamma.$
Take a weakly 2-summable  family $(x_n)_1^\infty\in X.$
It is enough (by a theorem of E. Landau) to show that for every sequence $a=(a_k)\in l_2^\infty$
$$\sum_{k=1}^\infty ||Tx_k||\, |a_k|<\infty.$$

Fix $a=(a_k)\in l_2^\infty$ and  
consider the series
\begin{equation}
\sum_{k=1}^\infty a_kx_k\gamma_k.
 \label{s1}
\end{equation}
Since for $n,m\in \mathbb N$
$$
||\sum_{k=n}^{n+m} a_kx_k\gamma_k||_{C(G,X)}=\sup_{s\in G} \sup_{||x'||\le1}|\<\sum_{k=n}^{n+m} a_kx_k\gamma_k(s), x'\>|\le$$
$$\sqrt{\sum_{k=n}^{n+m}|a_k|^2 }\sqrt{ \sup_{||x'||\le1} \sum_{k=1}^{\infty}|\<x_k, x'\>|^2}
$$
and the space $C(G,X)$ is complete, the series (\ref{s1}) converges in $C(G,X).$

Now fix $s\in G$ and take an $x'\in X^*$ with $||x'||\le 1.$
For the operator $u: l_2\to C(G,X),$
defined by $ub:= \sum_{k=1}^\infty b_kx_k\gamma_k$ for $b:=(b_k)\in l_2$ we have
$$
|\sum_{k=1}^\infty b_k \<x_k,x'\>\gamma_k(s)|^2\le ||b||^2 \sum_{k=1}^\infty |\<x_k,x'\>|^2.
$$
Hence,
$$
||\sum_{k=1}^\infty b_k x_k\gamma_k(s)||^2_X\le \sup_{||x'||\le1} ||b||^2 \sum_{k=1}^\infty |\<x_k,x'\>|^2
$$
and $$||ub||^2= \sup_{s\in G} ||\sum_{k=1}^\infty b_k x_k\gamma_k(s)||^2_X \le ||b||^2 \sup_{||x'||\le1} \sum_{k=1}^\infty |\<x_k,x'\>|^2.
$$
Therefore,
  
  \begin{equation}
\begin{aligned}
||u||^2&\le \sup_{||x'||\le1} \sum_{k=1}^\infty |\<x_k,x'\>|^2 \\ 
&\le  \sup_{\overset{F\in C(G,X)^*}{||F||\le1}} \sum_{k=1}^\infty |\<x_k\gamma_k,F\>|^2 =||(x_k\gamma_k)||_2^{weak}=||u||^2.
\end{aligned}
\end{equation}
  
Thus, we have:
\begin{equation}
||(x_k)||_2^{weak}
=||u||  \label{wl2}
\end{equation}
    
Now, for our fixed sequence $a=(a_k)\in l_2$ define an operator $A: M(G)\to l_2$ by
$$
A\mu:=(a_k\widehat{\mu}(\gamma_k)),\, \mu\in M(G).
$$
For $s\in G,$ put $\ove f(s):= \sum_{n=1}^\infty a_nx_n\gamma_n(s).$ Then $\ove f\in C(G,X)$ and
$$
uA\mu= \sum_{k=1}^\infty a_k\widehat{\mu}(\gamma_k) x_k\gamma_k=
\star\ove f(\mu).
$$
By assumption, the family $(T\widehat{\ove f})$ is absolutely summable. It means that
$(a_nTx_n)\in l_1(Y)$ and this is true for any sequence $a\in l_2.$
Therefore, $(Tx_n)\in l_2(Y).$ 

$3) \implies 1).$\,
We may (and do) assume that $X$ is separable (since the subspace $\ove f(G)\subset X$ is separable).

Let $\star {\ove f}=BA$ be a continuous factorization of $\star {\ove f}$ through a Hilbert space:
$$
\star \ove f: M(G)\overset A\to H\overset B\to C(G,X)=C(G)\wt\ot X.
$$
Note that, for $s\in G,$ $\gamma\in \Gamma,$
$$
\star \ove f(\gamma dt)(s)= \int \ove f(t)\gamma(s-t)\, dt=\int \ove f(t)\gamma(s)\ove{\gamma(t)}\, dt=
\gamma(s)\wh{\ove f}(\gamma).
$$
Also, $T\circ \star \ove f$ maps $\gamma$ to $\gamma T(\wh{\ove f}(\gamma))$ (we identify $\gamma dt$ with $\gamma).$
Let
$$
x_\gamma:= \wh{\ove f}(\gamma),\, y_\gamma:= Tx_\gamma\ \, \text{ (so }\ \gamma y_\gamma:=\gamma\ot y_\gamma\in C(G)\wt\ot Y).
$$
If $x'\in X^*$ and $i: C(G)\to C(G)$ is the identity map, 
then $\<x', \ove f(\cdot)\>\in C(G)$ and we have a diagram:
$$
 x'\circ BA=x'\circ \star \ove f: M(G)\overset A\to H\overset B\to C(G,X)=C(G)\wt\ot X\overset{i\ot x'}\longrightarrow C(G)
$$
with $ x'\circ BA(\gamma)= \star \ove f(\gamma)= \<x', x_{\gamma}\>\gamma.$
By the Pisier's result, for every $x'\in X$ the series $\sum_\gamma |\<x', x_{\gamma}\>|$ is convergent.
Hence, only countably many of $x_\gamma$'s are not zero (recall that $X$ is assumed to be separable).
Since $A\in \Pi_1(M(G), H),$ the mapping $\star \wh{\ove f}=BA$ is absolutely 2-summing. So,  the family
$(x_{\gamma})$ is stronly 2-summable (and countable).
Thus, $(x_{\gamma})\in l_1^{weak}(X)\cap l_2(X).$ By 3), $(y_\gamma)=T(x_{\gamma})\in l_1(Y).$
\end{proof}

It follows immediately:

\begin{corollary}\label{CorPi}
 If $T\in \Pi_1(X,Y),$ then $T\in L(X,Y)$ possesses the property $(\mathfrak l).$
 \end{corollary}
 
 Consider a generalization of Definirion \ref{D1}.
 
 \begin{definition}\label{D2}
 Let $0< s\le1$ and  $1/r=1/s-1.$
$T\in L(X,Y)$ possesses the property $(\mathfrak l_r)$ if
for every $\ove f\in C(G,X)$ such that $\star \ove f$ can be factored through an $S_r$-operator in a Hilbert space 
$$
\star \ove f: M(G)\overset A\to H\overset {U}\to H\overset B\to C(G,X)
$$
the family $T\widehat{\ove f}$ is absolutely s-summable.
\end{definition}
 
 The proof of the following result is the same as the proof of the implication $3)\implies 1)$
 of the previous theorem  (but instead of the Pisier's result one must use Theorem \ref{ThP1} for 
 the general case). Cf. Corollary \ref{CorPi}.
 
\begin{theorem}\label{ThS3}
For functions $\ove f\in C(G,X),$
consider the convolution operators
$\star \ove f: M(G)\to C(G,X)$ and let 
$T\in L(X,Y).$ Consider the following assertions:

1)\,
$T\in L(X,Y)$ possesses the property $(\mathfrak l_r).$

2)\,
$T\in \Pi_r(X,Y).$


We have:
$2) \implies 1).$
\end{theorem} 
 
 Now, a concrete corollary from the theorems.
 
 \begin{corollary}
 For $1\le p \le2$ it follows from $T\in \Pi_2(L_p(\nu), Y)$ that
 $T$ possesses the properties $(\mathfrak l_r)$ for all $r\in (0,1].$
 On the other hand, $T\in (\mathfrak l)$ implies $T\in \Pi_1(L_p(\nu), Y).$ 
  \end{corollary}
  
  Indeed, as is well known, for any $s\in (0,2]$ we have $\Pi_s(L_p(\nu),Y)=\Pi_2(L_p(\nu), Y)$
(see  \cite{PiOP}).  
  

\section{On a Pisier's result: Vector-valued case}\label{3}

In this section we  are going to give some generalizations of the Pisier's theorem mentioned
in Introduction to the cases of $S_{p}$-factorizations of operators for
vector-valued cases.
We will generalize also a result of P. Saab \cite{Sa8}, Theorem 4.2, where
it was shown that the Pisier's techniques in the scalar case can be
extended to the vector-valued case (factorizations of a vector valued convolutions through Hilbert spaces).
In the end of the section, we consider the factorizations through $S_{p,q}$-operators linear mappings between
tensor products of several Banach spaces.

 Let $f\in C(G)$ and $T\in L(X,Y).$ All operators under considerations are
 supposed to be not identically zero.

 Denote by $M(G,X)$ the Banach space of all regular Borel $X$-valued measures
 of bounded variation, C(G,X) the Banach space of all continuous $X$-valued functions
defined on $G$ equipped with the supremum norm.

Note that $$M(G)\otimes X\subset M(G)\widehat\otimes X\subset M(G,X),$$
 where $\widehat\otimes$ is the projective tensor product.

 Define a map $T_f:=T\circ \star f: M(G,X)\to C(G,Y)$ by
 $$
 T_f(\overline \mu)(s)= \int_G f(s -t)\, dT\overline\mu(t),\ \overline \mu\in M(G,X).
 $$

 \begin{theorem}\label{Th1}
Let $f\in C(G),$
$0<r,s<\infty.$
Consider a convolution operator
$\star f: M(G)\to C(G)$ and an operator $T: X\to Y.$
If the operator
$$T_f: M(G,X)\to C(G,Y)$$
can be factored through an $S_{r,s}$-operator then
the operators $\star f$ and $T$ possess the same property.
The same is true for the case where $r=s=\infty$
(or $q_1=q_2=1).$
\end{theorem}

\begin{proof}
We  may (and do) assume that $T\neq0$ and $f\neq0.$

We use partially (in the first part of the proof) an idea from \cite{Sa8}.
Let $T_f=BUA,$ where $A\in L(M(G,X),H),$ $U\in S_{r,s}(H)$ and $B\in L(H, C(G,Y).$
Fix a point $s_0\in G$ for which $f(s_0)\neq0.$
Define the operators $i: X\to M(G,X)$ and $j: C(G,Y)\to Y$ by
$ix=\delta_e\otimes x$ ($e$ is a neutral element of $G$) and
$jh=h(s_0)/f(s_0).$ For $s\in G$

\begin{equation}
\begin{aligned}
(T_fix)(s)&=\int f(s-t)\, dT(\delta_e\otimes x)(t)\\
&= \int_G f(s-t) Tx\, d\delta_e(t)=f(s)Tx\in C(G,Y).
\end{aligned}
\end{equation}

Hence, $jT_fix= f(s_0)Tx/f(s_0)=Tx$ or $T=jBUAi.$

Now, let $k: M(G)\to M(G,X)$ be defined by $k\mu=\mu\otimes x_0,$ where $x_0$ is such that
$||Tx_0||=1.$
Then $BUAk\mu(M(G))=:C_1\subset C(G,Y)$ and $UAk(M(G))=H_1\subset H.$
Denote by $P$ an orthogonal projector from $H$ onto$H_1$ and by $R$ the composition $Bl,$
where $l: H_1\to H$ is the identity injection. We have a diagram:
$$
M(G)\overset k\to M(G,X)\overset A\to H\overset U\to H\overset P\to H_1\overset R\to C_1\subset C(G,Y).
$$
If $\mu\in M(G),$ then
$$RPUAk\mu=T_fk\mu= T_f(\mu\otimes x_0)= Tx_0 \int_G f(\cdot-t)\, d\mu(t)=Tx_0\, f*\mu.
$$
Therefore, $C_1= \{Tx_0\, f*\mu:\, \mu\in M(G)\}\subset C(G)\otimes  \operatorname{span}\{Tx_0\}\subset C(G,Y).$
Take a functional $y'\in Y^*$ with $\langle y',Tx_o\rangle =||y'||=1$ and define an operator
$V: C(G)\otimes  \operatorname{span}\{Tx_0\}\to C(G)$ putting
$V(h\otimes y)= h\langle y', y\rangle $ for $y\in \operatorname{span}\{Tx_0\}.$
Then, for $\mu\in M(G),$
$$
VRPUAk\mu= V(f*\mu\otimes Tx_0)=f*\mu.
$$
Thus, the convolution operator $\star f$ is factorized through an operator from $S_{r,s}.$
\end{proof}

\begin{remark}\label{Rem1}
It is clear from the proof that the condition
"the operator $T_f: M(G,X)\to C(G,Y)$" can be changed by the condition
"the restricted operator $T_f: M(G)\widehat\otimes X)\to C(G,Y).$"
\end{remark}

\begin{corollary}
 Let $f\in C(G),$
$0<p\le\infty.$
Consider a convolution operator
$\star f: M(G)\to C(G)$ and an operator $T: X\to Y.$
If the operator
$$T_f: M(G,X)\to C(G,Y)$$
can be factored through an $S_{p}$-operator then
the operators $\star f$ and $T$ possess the same property.
\end{corollary}

 In a partial case where $X=Y,$ $T=\operatorname{id}_X$ and $p:=p_1=p_2=\infty$ we get
 a result of E. Saab \cite{Sa8}:

 \begin{corollary}
 Let $f\in C(G),$
Consider a Banach space $X$ and a convolution operator
$\star f: M(G)\to C(G).$
If the operator
$$T_f: M(G,X)\to C(G,X)$$
can be factored through a Hilbert space then
  $\hat f\in l_1$ and $X\cong H.$
\end{corollary}

We proof now a general theorem:

 \begin{theorem}\label{Th2}
Let $0<s\le r<\infty,$ $T_i\in L(X_i,Y_{i})$ for $i=1,2,\dots, m.$
If the operators $T_i$
can be factored through the $S_{r,s}$-operators then
the tensor product
$$T:= T_1\otimes T_2\otimes \dots \otimes T_m: X_1\widehat\otimes X_2\widehat\otimes\dots\widehat\otimes X_m\to Y_1\widetilde\otimes Y_2\widetilde\otimes\dots\widetilde\otimes Y_m$$
  possesses the same property.
\end{theorem}
\begin{proof}
It is enough to  consider the case of the product of two operators.
Let
$$
T_i: X_i\overset {A_i}\to H_i\overset {U_i}\to H_i\overset {B_i}\to Y_i
$$
be the factorizations of operators $T_i, i=1,2,$
Here $A_i\in L(X_i,H_i), B_i\in L(H_i,Y_i)$ and $U_i\in S_{s,r}(H_i)$ for $i=1.2.$
Let $U_i:= \sum_{n=1}^{\infty} s^i_n e_n^i\otimes f_n^i,$ where $(e^i_n), (f^i_n)$ are
orthonormal systems in corresponding Hilbert spaces and $(s_n^i)\in l_{s,r}$ are the sequences
of singular numbers of the operators $U_i$ $(i=1,2.)$
Tensor product $U_1\otimes U_2$ has the following representation:

\begin{equation}
\begin{aligned}
U_1\otimes U_2(h_1\otimes h_2) &= \sum_{n=1}^{\infty} s^1_n (h_1,e_n^1) f_n^1\otimes
\sum_{k=1}^{\infty} s^2_k (h_2,e_k^2) f_k^2\\
&= \sum_{k,n} s^1_n s^2_k (h_1\otimes h_2, e^1_n\otimes e^2_k) f^1_n\otimes f^2_k.
\end{aligned}
\end{equation}

The last series is convergent in $H_1\otimes_2 H_2$ since it  follows from the conditions on $r,s$ that, e. g.,
$U_1\otimes U_2\in S_{2r}$ (we have $\sum_{k,n} (s^1_n)^{2r} (s^2_k)^{2r}<\infty).$
Therefore, the sequence $(s^1_n s^2_k)_{k,n}$ is the sequence of all singular numbers of $U_1\otimes U_2.$
 Thus, $U_1\otimes U_2= \sum_{k,n} s^1_n s^2_k  (e^1_n\otimes e^2_k)\otimes (f^1_n\otimes f^2_k).$
 Since the sequences $(s^1_n)$ and $(s^2_k)$ belong to $l_{s,r}$ and $r\le s,$
 their product $( s^1_n s^2_k)$ is also in $l_{s,r}$ by the O'Neil's theorem (see~\cite{oneil}, Theorem 7.7).

Now, consider the mappings $A_1\otimes A_2$ and $B_1\otimes B_2.$
The first one acts from the projective tensor product $X_1\widehat\otimes X_2$ to the projective tensor product
$H_1\widehat\otimes H_2.$ The second one maps 
$H_1\widetilde\otimes H_2$ to $Y_1\widetilde\otimes Y_2.$ Denoting by $\varphi$ and $\psi$ the canonical injections
$H_1\widehat\otimes H_2\to H_1\otimes_2 H_2$ and $H_1\otimes_2 H_2\to H_1\widetilde\otimes H_2$ respectively, we obtain
a factorization  of $T_1\otimes T_2$ through an $S_{r,s}-operator:$
$T=B_1\otimes B_2 \psi U_1\otimes U_2 \varphi A_1\otimes A_2:$
$$
T: X_1\widehat\otimes X_2\overset{A_1\otimes A_2}\longrightarrow H_1\widehat\otimes H_2 \overset {\varphi}\to H_1\otimes_2 H_2
\overset{U_1\otimes U_2}\longrightarrow H_1\otimes_2 H_2\overset{B_1\otimes B_2}\longrightarrow Y_1\widetilde\otimes Y_2.
$$
\end{proof}

\begin{remark}
The converse of the theorem is also true (cf. the proof of Theorem \ref{Th1}).
\end{remark}

For the case where one of the space is $M(G)$ we can get a more general result. Recall that
if the space $X$ has the RN property, then $M(G, X)= M(G)\widehat\otimes X.$ This is rather simple:
Let $\overline \mu\in M(G,X.)$ If $X\in RN,$ then there is a function $\overline f\in L^1(G,|\overline\mu|; X)$
such that $\overline\mu(E)=\int_E \overline f\, d|\overline\mu|$ for every Borel set $E.$
Identifying the space $L^1(G,|\overline\mu|; X)$ with a
subspace of $M(G, X)$ in a natural way,
we see that $\overline\mu\in L^1(G,|\overline\mu|; X)=L^1(G,|\overline\mu|)\widehat\otimes X$
(see~\cite{DiVM}).

Thus, it follows from the theorem above that if $0<s\le r<\infty$ and $X\in RN,$ then the possibility of factorization through
$S_{s,r}$-operators of the operators $\star f: M(G)\to C(G)$ and $T: X\to Y$ implies the possibility
of such a factorization for the operator $T_f: M(G, X)\to C(G,Y).$ However, we can prove such
a theorem without any assumption on the Banach space $X.$

Below we will use the following simple fact: If an operator $S: Z\to W$ in Banach spaces
can be factored as
$$
S: Z\overset L\to H\overset V\to H\overset M\to W
$$
and $W_0:=\overline{S(Z)}\subset W,$ then there is an operator $M_0: H\to W_0$ such that
$S$ has the factorization
$$
S: Z\overset L\to H\overset V\to H\overset{M_0}\to W_0\overset j\hookrightarrow W,
$$
where $j$ is an inclusion. Indeed, consider the subspace $H_0:=\overline{VL(Z)}\subset H,$
take an orthonormal projector $P: H\to H_0.$  
Put $M_0:= M|_{H_0}PVL.$

\begin{theorem}\label{ThfT}
Let $f\in C(G),$  
$0<s\le r<\infty.$
Consider a convolution operator
$\star f: M(G)\to C(G)$ and an operator $T: X\to Y.$
If the operators $\star f$ and $T$
can be factored through the $S_{r,s}$-operators then
the operators $$T_f: M(G,X)\to C(G,Y)$$ possesses the same property.
\end{theorem}

\begin{proof}
Denote the restriction of the operator $T_f$ onto $M(G)\widehat\otimes X$ by $\widetilde T_f$
We have:
$$
M(G,X)=I(C(G),X)\  \text{ and }\ (X^*\widetilde\otimes C(G))^*=I(C(G), X^{**})\supset M(G, X)
$$
(for the first equality, see \ref{DiVM} (Diestel-Uhl, Vector Measures, p. 162, Th. 3)).

By Theorem \ref{Th2}, the restricted operator $\widetilde T_f: M(G)\widehat\otimes X\to C(G,Y)$
can be factored through a $S_{r,s}$-operator. Then the dual operator
$$
\widetilde T_f^*:\,  I(Y,M(G))= (C(G)\widetilde\otimes Y)^*\to L(X, C(G)^{**})
$$
can be factored through a $S_{r,s}$-operator (we use the equality $C(G,Y)=C(G)\widetilde\otimes Y).$
But
$$Y^*\widehat\otimes M(G)\subset I(Y, M(G))$$
since the space $C(G)$ and all of  its duals have the metric approximation property.
Therefore,
$\widetilde T_f^*$ maps $Y^*\widehat\otimes M(G)$ into $X^*\widetilde\otimes C(G)$
(apply definition of $T_f)$ and its restriction to the first tensor product
can be factored through a $S_{r,s}$-operator.

Let $\tau$ be a restriction of  $\widetilde T_f^*$ to the subspace $Y^*\widehat\otimes M(G).$
Consider the dual operator $\tau^*:$
$$
\tau^*: I(C(G), X^{**})\to L(M(G), Y^{**}).
$$
Since $I(C(G),Z)=\Pi_1(C(G),Z)$ for any Banach space $Z$ and the ideal $\Pi_1$ of
$1$-absolutely summing operators is  injective,     
the space $M(G,X)$ can be naturally identify with a subspace of $I(C(G), X^{**})$
and the restriction of $\tau^*$ to this subspace is nothing that $T_f.$
It follows that $T_f$ can be factored through a $S_{r,s}$-operator.
\end{proof}

\begin{corollary}
Let $f\in C(G),$
$0<s\le r<\infty.$
Consider a convolution operator
$\star f: M(G)\to C(G)$ and an operators $T_k: X_k\to Y_k,$ $k=1,2,\dots n,$
If the operators $\star f$ and $T_k$
can be factored through the $S_{r,s}$-operators then
the corresponding operator $$T_f: M(G,\widehat\bigotimes_{k=1}^n  X_k)\to C(G,\widetilde\bigotimes_{k=1}^n Y_k)$$ possesses the same property.
\end{corollary}
\begin{proof}
Apply Theorem \ref{Th2} to   $X=\widehat\otimes_{k=1}^n  X_k$ and to the tensor product of the operators $T_k.$
Then apply Theorem \ref{ThfT}.
\end{proof}

\begin{corollary}
Let $f_k\in C(G),$  $k=1,2,\dots n,$
$0<s\le r<\infty.$
Consider the convolution operators
$\star f_k: M(G)\to C(G)$ and an operators $T_k: X_k\to Y_k,$ $k=1,2,\dots n,$
If the operators $\star f_k$ and $T_k$
can be factored through the $S_{r,s}$-operators then
the corresponding operator $$T_f: \widehat\bigotimes_{k=1}^n M(G,  X_k)\to \widetilde\bigotimes_{k=1}^n C(G, Y_k)
$$ possesses the same property.
\end{corollary}
\begin{proof}
By Theorem \ref{ThfT}, for every $k$ the operator $T_{f_k}: M(G,X_k)\to C(G,X_k)$
can be factored through an $S_{r,s}$-operator.
By Theorem \ref{Th2}, the operator $T_f$ possesses the same property.
\end{proof}

\begin{remark}
The conditions of type $s\le r$ are essential (see the end of  Appendix).
\end{remark}

\newpage

\section{Appendix}

\subsection{On Lorentz sequence spaces}

Main facts here are taken from \cite{Pie82}.

\begin{theorem}
Let $p,q\in  (0,\infty)$ and $q>p.$
There exist $x,y\in l_{p,q}$ such that $x\ot y\notin l_{p,q}.$
\end{theorem}
This will be prooved below,

We need some more information on Lorentz spaces (see \cite{EigPie} for details).

The n-th {\it approximation number}\, of  $x\in l_\infty(I):$
$$
a_n(x):=\inf\{||x-u||_{l+\infty(I)}:\, u\in l_{\infty}(I), \text{card} (u)<n\}.
$$

$$
a_n(x):=\inf\{c\ge 0:\, \text{card} (i\in I: |x_i|\ge c)<n\}.
$$ 
When $x=(x_n)_{n=1}^\infty$ and $|x_1|\ge|x_2|\ge\cdots\ge0$ we have
$a_n(x)=|x_n|.$ 
In general case, $(a_n(x))$ is the non-increasing  
rearrangement of $x.$ An usual notation: $(x_n^*).$

The Lorentz space $l_{r,w}(I)$ consists of all complex-valued families $x=(x_i)$
such that $$
(n^{1/r-1/w} a_n(x))\in l_w.
$$
$l_{r,w}(I)$ is a quasi-normed space:
$$
||x||_{l_{r,w}}:= \(\sum_{n=1}^\infty n^{w/r-1}a_n(x)^w\)^{1/w}\ \text{ if }\ 0<w<\infty
$$
and 
$$
||x||_{l_{r,\infty}}:=\sup \{n^{1/r}a_n(x):\, n\in \mathbb N\}.
$$
If $I=\mathbb N,$ then we use the notation $l_{r,w}.$

From \cite{EigPie}(2.1.10):
\begin{lemma}
$$x\in l_{r,w}\ \text{ if and only if }\ (2^{k/r}a_{2^k}(x))\in l_w.
$$
$||(2^{k/r}a_{2^k}(x))||_{l_w}$ is an equivalent quasi-norm.
\end{lemma}
Moreover, we have the following from \cite{Pie82}
(getting a more general result).


A sequence $(n_k)$ of natural numbers beginning with
$n_0=1$ is called quasi-geometric if there are constants $a$ and $b$ such that
$$
1< a\le \frac{n_{k+1}}{n_k}\le b>\infty\  \text{ for }\ k=0,1,2,\dots.
$$
For example,
$$
n_k:=(k+1)2^k.
$$

\begin{lemma}\label{LA2}
Let $(n_k)$ be a quasi-geometric sequence and $x\in l_\infty(I).$
$$x\in l_{r,w}(I)\ \text{ if and only if }\ (n_k^{1/r}a_{n_k}(x))\in l_w.
$$
\end{lemma}

{\bf Notation:}\,
$$l_{r,u}\otimes l_{r,u}:=\{h\in l_\infty(\mathbb N\times\mathbb N):\, 
\exists\, x,y\in l_{r,u},\, h(n,m)=(x_ny_m)\}.$$
Also, $h=x\ot y.$

\begin{remark}
If $$l_{r,u}\otimes l_{r,v}\subset l_{r,u},$$
then there is a constant $c=c(r,u,v)"0$ so that for $x\in l_{r,u}$ and $y\in l_{r,v}$
$$
||x\ot y||_{l(r,u)}\le c ||x||_{r,u}||y||_{r,v}.
$$
\end{remark}


\begin{proposition}\label{PrA}
Let $r,u\in  (0,\infty).$ If
$$l_{r,u}(\mathbb N)\otimes l_{r,u}(\mathbb N)\subset l_{r,u}(\mathbb N\times\mathbb N),$$
then $u\le r.$ 
\end{proposition}
\begin{proof} Sketch:
Consider a partial case where $r=1, u=2$ (trying get a contradiction).

Take the finite sequences $X_m:=(x_{mn})$ defined by
$$
x_{mn}:= 2^{-i}\ \text{ if }\ 2^i\le n< 2^{i+1}\ \text{ and } i\le m
$$
and $x_{mn}=0$ otherwise.

It turns out that
$$
||x_m||_{l_{1,2}}\asymp m^{1/2}.
$$

Recall that the natural numbers $n_k:=(k+1)2^k$ constitute a quasi-geometric
sequence. 
Since the double sequences $x_m\ot x_m$ contain $n_k$-times the coordinate $2^{-k}$
whenever $k\le m,$ we obtain
$$
||x_m\ot x_m||_{l_{r,u}(N\times N)}\asymp \(\sum_{k=0}^\infty [n_ka_{n_k}(x_m\ot x_m)]^2\)^{1/2}
$$
$$
\asymp \(\sum_{k=0}^\infty (k+1)^2\)^{1/2}\asymp m^{3/2}.
$$
On the other hand, we must have
$$
||x\ot y||_{l(1,2)}\le c ||x||_{1,2}||y||_{1,2}.
$$
This yields
$$
m^{3/2}\prec m^{1/2+1/2}.
$$
\end{proof}

\subsection{Returning to convolutions}.

Almost the same proof as  the proof of Theorem \ref{ThP1} gives us

\begin{theorem}\label{ThP2}
Let $f\in C(G),$  $0<q, s\le1$ and  $1/r=1/s-1, 1/p= 1/q-1.$ Consider a convolution operator  
$\star f: M(G)\to C(G).$
If the operator $\star f$ can be factored through a
Lorentz-Schatten $S_{r,p}$-operator in a Hilbert space, then
the set  of  Fourier coefficients $\hat f$ belongs to $l_{s,q}.$
\end{theorem}

Also, for the tensor products we have (as a consequence of Theorem \ref{Th2}):

 \begin{theorem}\label{ThA2}
Let $0<s\le r<\infty,$ $f_1, f_2\in C(G).$
Consider the convolution operators $\star f_1, \star f_2: M(G)\to C(G).$
If these operators 
can be factored through the $S_{r,s}$-operators then
the tensor product
$$T:=\star f_1\otimes \star f_2: M(G)\widehat\otimes M(G)\to C(G)\widetilde\otimes C(G)$$
  possesses the same property. The converse is also true.
\end{theorem}

Take now $0<q, s<1$ and $r,u\in (0,\infty), u>r$  with $1/r=1/s-1, 1/u= 1/q-1.$
Using Proposition \ref{PrA},
take two function $f_1,f_2\in C(G)$ with $\wh f_1, \wh f_2\in l_{s,q}$ but
$\wh f_1\ot \wh f_2\notin l_{s,q}.$

Then the tensor product $T:=\star f_1\otimes \star f_2$ can not be factored through $S_{r,u}$-operator
since else $\wh f_1\ot \wh f_2\in l_{s,q}$ (the proof is similar to the proof of \ref{ThP2}).

Finally, note that all the fact where we have the restrictions of type $s\le r$ (cf. Theorem  \ref{ThP2})
are sharp (i.e., the conditions of type $s\le r$ are essential). 


\end{document}